\documentclass[11pt]{amsart}

\usepackage[margin=1in]{geometry}
\usepackage{amsmath}
\usepackage{float}
\usepackage{booktabs}
\usepackage[dvipsnames]{xcolor} 

\usepackage{fix-cm}

\usepackage{ wasysym }

\usepackage{blkarray}

\usepackage{amscd,amsmath,amssymb,amsfonts,amsthm,ascmac}
\usepackage{enumerate,mathrsfs,stmaryrd,latexsym, comment, mathtools, mathdots} 

\usepackage[mathscr]{euscript}

\usepackage{caption}
\usepackage{subcaption}
\usepackage{hyperref}

\usepackage[all]{xy}

\usepackage{tikz-cd}
\usepackage{tikz}
\usetikzlibrary{snakes, 
	3d, matrix,decorations.pathreplacing,calc,decorations.pathmorphing, patterns}
\usetikzlibrary {positioning}
\usepackage{tikz-3dplot}
\usetikzlibrary{calc,backgrounds}

\pgfdeclarelayer{background}
\pgfdeclarelayer{foreground}
\pgfsetlayers{background,main,foreground}

\def\aeMarkRightAngle[size=#1](#2,#3,#4){
   \draw ($(#3)!#1!(#2)$) -- 
         ($($(#3)!#1!(#2)$)!#1!90:(#2)$) --
         ($(#3)!#1!(#4)$);}

\usepackage{tikz-3dplot}
\usepackage{stackengine}

\numberwithin{equation}{section}
\allowdisplaybreaks[1]

\makeatletter
\newcommand{\leqnomode}{\tagsleft@true\let\veqno\@@leqno}
\newcommand{\reqnomode}{\tagsleft@false\let\veqno\@@eqno}
\makeatother

\DeclareMathOperator{\SL}{SL}

\DeclareMathOperator{\Gr}{Gr}

\DeclareMathOperator{\Hom}{Hom}
\DeclareMathOperator{\Sym}{Sym}

\DeclareMathOperator{\Sp}{Sp}
\DeclareMathOperator{\Spin}{Spin}
\DeclareMathOperator{\SO}{SO}
\DeclareMathOperator{\LGr}{LGr}
\DeclareMathOperator{\Aut}{Aut}
\DeclareMathOperator{\PGL}{PGL}
\DeclareMathOperator{\Pic}{Pic}

\newtheorem{theorem}{Theorem}[section]
\newtheorem{lemma}[theorem]{Lemma}
\newtheorem{proposition}[theorem]{Proposition}
\newtheorem{corollary}[theorem]{Corollary}

\theoremstyle{definition}
\newtheorem{example}[theorem]{Example}
\newtheorem{definition}[theorem]{Definition}

\begin{document}
	
\title{Deformation rigidity of the double Cayley Grassmannian}

\author{Shin-young Kim}
\address{Yonsei University, 50 Yonsei-ro Seodaemun-gu, Seoul, Republic of Korea, 03722}
\email{shin-young.kim@yonsei.ac.kr}

\author{Kyeong-Dong Park}
\address{Department of Mathematics and Research Institute of Molecular Alchemy, Gyeongsang National University, Jinju, Republic of Korea, 52828}
\email{kdpark@gnu.ac.kr}

\keywords{double Cayley Grassmannian, deformation rigidity, variety of minimal rational tangents, prolongations of a linear Lie algebra}
\subjclass[2010]{Primary: 14M27, 32G05, Secondary: 14M17, 32M12} 

\begin{abstract}
The double Cayley Grassmannian is a unique smooth equivariant completion with Picard number one of the 14-dimensional exceptional complex Lie group $G_2$, 
and it parametrizes eight-dimensional isotropic subalgebras of the complexified bi-octonions. 
We show the rigidity of the double Cayley Grassmannian under K\"{a}hler deformations. 
This means that for any smooth projective family of complex manifolds over a connected base of which one fiber is biholomorphic to the double Cayley Grassmannian, all other fibers are biholomorphic to the double Cayley Grassmannian.

\end{abstract}
\maketitle
\setcounter{tocdepth}{1} 

\date{\today}

\section{Introduction}

Let $G$ be a connected reductive algebraic group over $\mathbb C$ and let $\theta$ be a group involution of $G$. 
When $H$ is a closed subgroup of $G$ such that {$G^{\theta,0} \subset H \subset N_G(G^{\theta})$}, we call the homogeneous space $G/H$ \emph{a symmetric homogeneous space}. 
We call $X$ \emph{a symmetric variety} if $X$ is a normal $G$-variety together with a $G$-equivariant open embedding of a symmetric homogeneous space $G/H$ into $X$. 

By using the Luna--Vust theory on spherical varieties, A. Ruzzi \cite{Ruzzi2011} classified the smooth projective symmetric varieties with Picard number one via colored fans. 
As a result, there are only 6 non-homogeneous smooth projective symmetric varieties with Picard number one, whose restricted root system is of type $A_2$ or of type $G_2$. 
When smooth projective symmetric varieties with Picard number one have the restricted root system of type $A_2$, A. Ruzzi also gave us geometric descriptions in Theorem 3 of \cite{Ruzzi2010}: 
these $A_2$-type symmetric varieties are smooth equivariant completions of symmetric homogeneous spaces 
$\SL(3, \mathbb C)/\SO(3, \mathbb C)$, 
$(\SL(3, \mathbb C) \times \SL(3, \mathbb C))/\SL(3, \mathbb C)$, 
$\SL(6, \mathbb C)/\Sp(6, \mathbb C)$, $E_6/F_4$, 
and they are isomorphic to a general hyperplane section of rational homogeneous manifolds 
which are in the third row of the \emph{geometric Freudenthal--Tits magic square}: 
\begin{center}
\begin{tabular}{ c | c c c c }
\hline
   & $\mathbb R$ & $\mathbb C$ & $\mathbb H$ & $\mathbb O$  \\
 \hline
  $\mathbb R$ & $v_4(\mathbb P^1)$ & $\mathbb P (T_{\mathbb P^2})$ & $\Gr_{\omega}(2, 6)$ & $\mathbb{OP}^2_0$  \\
  $\mathbb C$ & $v_2(\mathbb P^2)$ & $\mathbb P^2 \times \mathbb P^2$ & $\Gr(2, 6)$ & $\mathbb{OP}^2$ \\
  $\mathbb H$ & $\LGr(3, 6)$ & $\Gr(3, 6)$ & $\mathbb S_6$ & $E_7/P_7$ \\
  $\mathbb O$ & $F_4^{ad}$ & $E_6^{ad}$ & $E_7^{ad}$ & $E_8^{ad}$ \\
\hline
\end{tabular}
\end{center}
On the other hand, by Theorem 2 of \cite{Ruzzi2010}, the smooth projective symmetric varieties with Picard number one whose restricted root system is of type $G_2$ are smooth equivariant completions of either $G_2/(\SL(2, \mathbb C) \times \SL(2, \mathbb C))$ or $(G_2 \times G_2)/G_2$. 
They are called the \emph{Cayley Grassmannian} and the \emph{double Cayley Grassmannian}, respectively, which are studied by L. Manivel \cite{Manivel18, Manivel19}. 
The double Cayley Grassmannian is a unique smooth equivariant completion with Picard number one of the 14-dimensional complex Lie group $G_2$. 
It can be described as the zero locus of a general global section of a certain vector bundle on the 21-dimensional spinor variety by Theorem~1 of \cite{Manivel19} (see Proposition~\ref{Description of the double Cayley Grassmannian} for details).  

A smooth projective variety $X$ is called \emph{globally rigid} if, for any smooth projective family over a connected base with one fiber isomorphic to $X$, all fibers of the family are isomorphic to $X$. For a rational homogeneous manifold $G/P$ with Picard number one, the global deformation rigidity was studied by Hwang and Mok
in \cite{Hw97}, \cite{HwM98}, \cite{HM02}, \cite{HwM04b}, \cite{HwM05}:  
a rational homogeneous manifold with Picard number one, 
different from the orthogonal isotropic Grassmannian $\text{OGr}(2, 7)=B_3/P_2$, is globally rigid. 
This result can be generalized to some quasi-homogeneous varieties with Picard number one. 
Among smooth projective horospherical varieties with Picard number one, the deformation rigidity of odd Lagrangian Grassmannians \cite{Park}, odd symplectic Grassmannians \cite{HL2019} and the $G_2$-horospherical manifold \cite{HL23} was studied. 

It is then natural to ask the same question about smooth projective symmetric varieties, and the problem was studied in \cite{KP19}. 
Recently, in \cite{ChenFuLi}, they have proved the global rigidity of smooth projective symmetric varieties with Picard number one of which the restricted root system is of type $A_2$. 
The aim of this paper is to prove the \emph{global rigidity under K\"ahler deformation} of the double Cayley Grassmannian.

\begin{theorem}\label{Main theorem} 
Let $\pi \colon \mathcal X \rightarrow \Delta$ be a smooth projective morphism from a complex manifold $\mathcal X$ to the unit disc $\Delta\subset\mathbb C$. 
Suppose for any $t\in \Delta \backslash \{0\}$,
the fiber $\mathcal X_t=\pi^{-1}(t)$ is biholomorphic to the double Cayley Grassmannian $S$.
Then the central fiber $\mathcal X_0$ is also biholomorphic to $S$.
\end{theorem}

To prove the global deformation rigidity of the double Cayley Grassmannian, we mainly use the VMRT theory developed by Hwang and Mok (see Section~3.1 for details), and the methods developed in \cite{KP19} and \cite{FuLi}: the prolongations of the cone structure defined by VMRT, and the reduction of deformation to the one of embedded toric surface.

In Section~2, we recall the description of the double Cayley Grassmannian as the zero locus of a general global section of a certain vector bundle on the 21-dimensional spinor variety, from which we can obtain the local deformation rigidity of complex structures of $S$. See also \cite{Manivel19} for this local rigidity. 
In Section~3, we will prove Theorem~\ref{Main theorem}. By considering the automorphism group of $S$ and the prolongations of the Lie algebra of infinitesimal automorphisms of the cone structure given by its VMRT, we know that the central fiber $\mathcal X_0$ is isomorphic to either $S$ or an equivariant compactification of the vector group $\mathbb G_a^{14}$ of dimension 14. 
If we assume that $\mathcal X_0$ is isomorphic to an equivariant compactification of the vector group $\mathbb G_a^{14}$, then from the same argument as in \cite{FuLi} we get a family of smooth projective surfaces of Picard number 4 such that its general fiber is the toric surface given by the closure of a maximal torus in $G_2$ and the central fiber is an equivariant compactification of the vector group $\mathbb G_a^{2}$ of dimension 2. 
The family of surfaces also admits a finite group action induced from the $G_2$-action on $\mathcal X$. 
As this action is transitive on the set of irreducible components of the boundary divisor, it gives two possible restrictions on the central fiber due to the classification on vector group compactifications of $\mathbb G_a^{2}$ in \cite{HT99}. 
However, their Picard numbers are 1 or 2, not 4, which gives a contradiction concluding the global deformation rigidity of the double Cayley Grassmannian.

\vskip 1em

\noindent
\textbf{Acknowledgements}.
The authors would like to thank Laurent Manivel
for beneficial discussions and inspiration, who kindly explained his ideas on this topic to the second author. 
They also would like to express their thanks to Baohua Fu, Qifeng Li, and Jaehyun Hong for valuable comments.

Shin-young Kim was supported by Basic Science Research Program through the National Research Foundation of Korea (NRF) funded by the Ministry of Education (RS-2021-NR060138), and by the National Research Foundation of Korea (NRF) grant funded by the Korean government (MSIT) (RS-2024-00341743).

Kyeong-Dong Park was supported by the National Research Foundation of Korea (NRF) grant funded by the Korea government (MSIT) (RS-2021-NR062093), and by the fund of research promotion program, Gyeongsang National University, 2022. 
He was also supported by Global - Learning \& Academic research institution for Master’s·PhD students, and Postdocs (LAMP) Program of the National Research Foundation of Korea (NRF) grant funded by the Ministry of Education (RS-2023-00301974). 

\section{Geometric description and local deformation rigidity} 

A. Ruzzi \cite{Ruzzi2011} proved that the symmetric homogeneous space $(G_2 \times G_2)/ G_2$ admits a unique smooth equivariant completion with Picard number one, which is called the \emph{double Cayley Grassmannian}. 
As a complex Lie group, $G_2$ is the algebra automorphism of the complexified octonions $\mathbb O_{\mathbb C}:= \mathbb C \otimes_{\mathbb R} \mathbb O$ {and} its Cartesian product $G_2 \times G_2$ acts on the complexified bi-octonions $\mathbb C \otimes_{\mathbb R} (\mathbb C \otimes_{\mathbb R} \mathbb O) \cong \mathbb O_{\mathbb C} \oplus \mathbb O_{\mathbb C}$. 
For any element $g \in G_2$, {there is} the corresponding graph $$\Gamma_g = \{ (x, gx) \in \mathbb O_{\mathbb C} \oplus \mathbb O_{\mathbb C} : x \in \mathbb O_{\mathbb C} \}$$ which is an eight-dimensional isotropic subalgebra of the complexified bi-octonions, with respect to the difference of the octonionic norms on the two copies of $\mathbb O_{\mathbb C}$. 
An eight-dimensional isotropic subalgebra transverse to the decomposition $\mathbb O_{\mathbb C} \oplus \mathbb O_{\mathbb C}$ is the graph $\Gamma_g$ for some $g \in G_2$. 
Hence, the homogeneous space $(G_2 \times G_2)/ G_2$ parametrizes transverse eight-dimensional isotropic subalgebras. 
In fact, the double Cayley Grassmannian parametrizes eight-dimensional isotropic subalgebras of the complexified bi-octonions by Proposition~6 of \cite{Manivel19}. 

Since the intersection of $\Gamma_g$ with $\mbox{Im}\, \mathbb O_{\mathbb C} \oplus \mbox{Im}\, \mathbb O_{\mathbb C} \cong \mathbb C^{14}$ is a seven-dimensional isotropic subspace of 14-dimensional orthogonal vector space, we can get an embedding of $G_2$ into the 21-dimensional spinor variety $\mathbb S_{14}=\Spin_{14}(\mathbb C)/P_7$, where {$P_7$} denotes the maximal parabolic subgroup associated to the simple root $\alpha_7$ of $D_7$. 
The double Cayley Grassmannian is the closure of $G_2$ under this embedding, so we can obtain a geometric description of the double Cayley Grassmannian as a subvariety of the spinor variety $\mathbb S_{14}$. 

\begin{proposition}[Theorem 1 of \cite{Manivel19}] 
\label{Description of the double Cayley Grassmannian} 
The double Cayley Grassmannian $S$ is projectively equivalent to the zero locus of a general global section of the rank seven vector bundle $\mathcal U \otimes L$ on the 21-dimensional spinor variety $\mathbb S_{14}=\Spin_{14}(\mathbb C)/P_7$, 
where $\mathcal U$ denotes the universal subbundle of rank seven on $\mathbb S_{14}$ and 
$L$ is the very ample line bundle defining the minimal embedding $\mathbb S_{14} \hookrightarrow \mathbb P^{63}$. 
\end{proposition}

Since the Fano index of the 21-dimensional spinor variety $\mathbb S_{14}$ is equal to 12, 
$$K_S = K_{\mathbb S_{14}} \otimes \det(\mathcal U \otimes L) = L^{-12} \otimes (\det \mathcal U) \otimes L^{7} = L^{-12} \otimes L^{-2} \otimes L^{7} = L^{-7}$$
by the adjunction formula and Proposition~\ref{Description of the double Cayley Grassmannian}. 

\begin{corollary}[Corollary 3 of \cite{Manivel19}]  
\label{Fano index}
The double Cayley Grassmannian $S$ is a prime Fano variety of dimension 14 and index 7. 
\end{corollary}

From the Kodaira--Spencer deformation theory (for example, see \cite{Kodaira86}),  
it suffices to prove the vanishing of the first cohomology group $H^1(S, T_S)$ for the local deformation rigidity of complex structures of a compact complex manifold $S$. From the geometric description of the double Cayley Grassmannian $S$ in Proposition~\ref{Description of the double Cayley Grassmannian}, we obtain the Koszul complex associated with a general global section of $\mathcal U \otimes L$. Using the Borel--Weil--Bott theorem from \cite{Bott57} with the normal exact sequence leads to the local deformation rigidity of $S$. 

\begin{proposition}[Proposition~5 of \cite{Manivel19}]
\label{Local deformation rigidity of the double Cayley Grassmannian} 
The complex structure of the double Cayley Grassmannian $S$ is locally rigid, that is, for any smooth projective family $p \colon \mathscr S \to B$ over a connected base $B$ with one fiber isomorphic to $S=p^{-1}(a)$, there exists an analytic open neighborhood $\mathcal N$ of $a\in B$ such that the fibers $p^{-1}(b)$ of the family are isomorphic to $S$ for all $b \in \mathcal N$. 
\end{proposition}

\section{Global deformation rigidity} 

\subsection{Deformation rigidity of variety of minimal rational tangents}
Let $X$ be a uniruled projective manifold. 
The \emph{variety of minimal rational tangents} is defined as the subvariety of the projectivized tangent bundle $\mathbb P (T_X)$ consisting of tangent directions of immersed minimal rational curves in a uniruled projective manifold $X$ (see \cite{HM99} and \cite{Hw00}).

By a \emph{parameterized rational curve} we mean that a non-constant holomorphic map $f \colon \mathbb P^1 \to X$ from the projective line $\mathbb P^1$ into $X$. 
We say that a (parameterized) rational curve $f \colon \mathbb P^1 \to X$ is \emph{free} if the pullback $f^* T_X$ of the tangent bundle is nonnegative. 
Let $\mathcal K$  be a connected component of the scheme of rational curves and we call $\mathcal K$ a {\it minimal rational component} if $\mathcal K_x$ is proper and nonempty for general $x \in X$. 
Let $f \colon \mathbb P^1 \rightarrow X$ be a minimal rational curve with $f(o)=x$ which is a member of $\mathcal K_x$ smooth at $x$.
Define the \emph{(rational) tangent  map} $\tau_x \colon \mathcal K_x \dashrightarrow \mathbb P(T_xX)$ by $\tau_x([f(\mathbb P^1)])=[df(T_o\mathbb P^1)]$ sending $[f]$ to its tangent direction at $x$.
For a general point $x \in X$, 
by Theorem 3.4 of \cite{Kebekus}, this tangent map induces a morphism 
$\tau_x \colon \mathcal K_x \rightarrow \mathbb P(T_xX)$
which is finite over its image.

\begin{definition}
Let $X$ be a uniruled projective manifold with a minimal rational component $\mathcal K$.
For a general point $x \in X$, the image
$\mathcal C_x(X) :=\tau_x(\mathcal K_x)\subset \mathbb P(T_xX)$ of the tangent map
is called the {\it variety of minimal rational tangents}
(to be abbreviated as VMRT) of $X$ at $x$. 
The union of $\mathcal C_x$ over general points $x\in X$ gives the fibration  $\mathcal C \subset \mathbb P(T_X)$ on $X$ except for a bad locus. 
This is called the {\it variety of minimal rational tangents associated to $\mathcal K$}. 
\end{definition}

For the wonderful compactification of $G_2$, its VMRT is known as the adjoint variety $G_2^{ad} \subset \mathbb P(\mathfrak g_2)$ by \cite{BF15}. 
Since the wonderful compactification of $G_2$ is obtained from the double Cayley Grassmannian $S$ by blowing up along the closed orbit, the VMRT of $S$ must contain the adjoint variety $G_2^{ad}$.
We note that the tangent space $T_s S$ for a general point $s \in S$ is identified with $\mathfrak g_2 = \mathbb C^{14}$ and the adjoint action of $G_2$ on $\mathfrak g_2$ is the isotropy representation. 

\begin{proposition}[Proposition~19 of \cite{Manivel19}]
\label{VMRT of S}
For a general point $s \in S$, the VMRT $\mathcal C_s(S)$ of the double Cayley Grassmannian $S$ is projectively equivalent to the adjoint variety $G_2/P_2 = G_2^{ad} \subset \mathbb P (\mathfrak g_2) =\mathbb P(T_s S)\cong \mathbb P^{13}$ of dimension five, 
which is the projectivization of the minimal nilpotent orbit under the adjoint action. 
Here, $P_2$ is the maximal parabolic subgroup associated with the long simple root $\alpha_2$ of the exceptional Lie group $G_2$. 
\end{proposition}

\begin{corollary} \label{VMRT}
Let $\pi \colon \mathcal X \rightarrow \Delta$ be a smooth projective morphism from a complex manifold $\mathcal X$
to the unit disc $\Delta\subset\mathbb C$. 
Suppose for any $t\in \Delta \backslash \{0\}$,
the fiber $\mathcal X_t=\pi^{-1}(t)$ is biholomorphic to the double Cayley Grassmannian $S$.
Then the VMRT of the central fiber $\mathcal X_0$ at a general point $x$ is projectively equivalent to the adjoint variety $G_2^{ad} \subset \mathbb P^{13}$.
\end{corollary}

\begin{proof}
Choose a section $\sigma \colon \Delta \rightarrow \mathcal X$ of $\pi$ such that $\sigma(0)=x$ and $\sigma(t)$ passes through a general point in $\mathcal X_t$ for $t \neq 0$. 
Let $\mathcal K_{\sigma(t)}$ be the normalized Chow space of minimal rational curves passing through $\sigma(t)$ in $\mathcal X_t$. 
Then the canonical map $\mu \colon \mathcal K_{\sigma} \rightarrow \Delta$ given by the family $\{ \mathcal K_{\sigma(t)} \}$ 
is smooth and projective by the same proof as that of Proposition 4 of \cite{HwM98}. 
The main theorem of \cite {Hw97} says that 
$\mathcal K_{\sigma(t)}$ is isomorphic to $\mathcal K_{s}(S)$ for all $t \in \Delta$ 
because $\mathcal K_{s}(S) \cong \mathcal C_{s}(S)$ is a homogeneous contact manifold $G_2^{ad}$ by Proposition \ref{VMRT of S}. 
Thus it suffices to show that the image of the tangent map for the central fiber is nondegenerate in $\mathbb P(T_x \mathcal X_0)$, that is, the image is not contained in any hyperplane of the projective space $\mathbb P(T_x \mathcal X_0)$.

Assume that the VMRT $\mathcal C_x (\mathcal X_0)$ at a general point $x$ is degenerate in $\mathbb P(T_x \mathcal X_0)$ and let $W$ be a distribution defined by the linear span $W_x$ of $\mathcal C_x (\mathcal X_0)$. 
The adjoint variety $G_2^{ad} \subset \mathbb P (\mathfrak g_2) \cong \mathbb P^{13}$ is linearly nondegenerate and hence, the variety of tangential lines $\mathfrak T_{\widehat G_2^{ad}} \subset \wedge^2 \mathfrak g_2$ is nondegenerate by Lemma 2.12 of \cite{FuLi}. 
Since $\mathcal K_{\sigma(t)}$ is a trivial family for all $t \in \Delta$, the variety of tangential lines $\mathfrak T_{\widehat{\mathcal C}_x (\mathcal X_0)}$ is nondegenerate in $\wedge^2 W_x$, which implies that the distribution $W$ is integrable by Proposition~9 of \cite{HwM98}. 
However, since $\mathcal C_x (\mathcal X_0)$ is the homogeneous contact manifold $G_2^{ad}$ of Picard number one, by Proposition~13 of \cite{HwM98}, the distribution $W$ defined by the linear span $W_x$ of $\mathcal C_x (\mathcal X_0)$ cannot be integrable which is a contradiction. 
Hence, the VMRT $\mathcal C_x (\mathcal X_0)$ at a general point $x$ is nondegenerate in $\mathbb P(T_x \mathcal X_0)$.
\end{proof}

\subsection{Isotrivial cone structures given by VMRT}

In this subsection, we focus on the isotrivial cone structure $\mathcal C$ given by VMRT on the central fiber $\mathcal X_0$. We mainly use the prolongation method to prove Proposition \ref{prop:two central fiber}.

\begin{definition}
Let $V$ be a complex vector space and
$\mathfrak g \subset \mathfrak{gl}(V)$ a linear Lie algebra.
For an integer $k\geq 0$, the space $\mathfrak g^{(k)}$,
called the \emph{k-th prolongation} of $\mathfrak g$, is
the vector space of symmetric multi-linear homomorphisms
$A \colon \Sym^{k+1}V \to V$ such that
for any fixed vectors $v_1, \cdots, v_k \in V$,
the endomorphism
$$v\in V \mapsto A(v_1,\cdots, v_k,v)\in V$$
belongs to the Lie algebra $\mathfrak g$.
That is, $\mathfrak g^{(k)}=\Hom(\Sym^{k+1} V, V)\cap\Hom(\Sym^k V, \mathfrak g)$.
\end{definition}

We are interested in the case where a Lie algebra $\mathfrak g$ is relevant to geometric contexts, 
in particular, the Lie algebra of infinitesimal linear automorphisms of the affine cone of an irreducible projective subvariety given by VMRT.

\begin{definition}
Let $Z \subset \mathbb P V$ be an irreducible projective variety.
The {\it projective automorphism group} of
$Z$ is $\Aut(Z)=\{ g\in \PGL(V) : g(Z)=Z \}$ and its Lie algebra is denoted by $\mathfrak{aut}(Z)$.
Denote by $\widehat Z \subset V$ the affine cone of $Z$. 
The {\it Lie algebra of infinitesimal linear automorphisms} of $\widehat{Z}$ is defined by
\begin{align*}
\mathfrak{aut}(\widehat{Z})&
=\{A\in \mathfrak{gl}(V) : \mbox{exp}(t A)(\widehat{Z})\subset \widehat{Z}, \, t\in \mathbb C \}, 
\end{align*}
where $\mbox{exp}(t A)$ denotes the one-parameter group of linear automorphisms of $V$.
Its $k$-th prolongation $\mathfrak{aut}(\widehat{Z})^{(k)}$ will be called the \emph{$k$-th prolongation of $Z \subset \mathbb P V$}.
\end{definition}

The \emph{Lie algebra $\mathfrak{aut}(\mathcal C, x)$ of infinitesimal automorphisms of the cone structure $\mathcal C$ at $x$} is the set of all germs of holomorphic vector fields preserving the cone structure $\mathcal C$ at $x$.

\begin{proposition}[Proposition 5.10 of \cite{FuHw}]\label{aut} 
Let $\mathcal C \subset \mathbb P(T_M)$ be a cone structure on a complex manifold $M$ and $x \in M$.
If $\mathfrak{aut}(\widehat{\mathcal C_x})^{(k+1)} = 0$ for some $k \geq 0$,
then $$\dim \mathfrak{aut}(\mathcal C, x) \leq
\dim M +\dim \mathfrak{aut}(\widehat{\mathcal C_x}) +\dim \mathfrak{aut}(\widehat{\mathcal C_x})^{(1)}+ \cdots + \dim \mathfrak{aut}(\widehat{\mathcal C_x})^{(k)}.$$
\end{proposition}

If the above equality holds, the isotrivial cone structure $\mathcal C$ given by VMRT on $M$ should be locally flat by Corollary~5.13 of \cite{FuHw} and $M$ is an equivariant compactification of the vector group of dimension $\dim M$ from Theorem~1.2 of \cite{FuHw2018-2}. 

\begin{definition}
A vector group $\mathbb G_a^n$ of dimension $n$ is the $n$-fold direct product of the additive group $\mathbb G_a$ of the complex number field $\mathbb C$. 
An \emph{equivariant compactification} of the vector group $\mathbb G_a^n$ is a smooth projective $\mathbb G_a^n$-variety with an open orbit isomorphic to $\mathbb G_a^n$.  
\end{definition}

\begin{example}[Propositions~5.2 and 5.5 of \cite{HT99}]\label{ex: rank 2 vector group} 
When $n=2$, there exist three kinds of the minimal model of equivariant compactifications of the vector group $\mathbb G_a^2$, that is $\mathbb P^2$, $\mathbb P^1 \times \mathbb P^1$, and a Hirzebruch surface $S(1,k)$ for $k>1$. 
For $\mathbb P^2$, there exist two different equivariant compactification structures. 
For $\mathbb P^1 \times \mathbb P^1$, there exists a unique equivariant compactification structure arising from the unique equivariant compactification structure of $\mathbb G_a^1$ on $\mathbb P^1$. 
For $S(1,k)$, there exist two different equivariant compactification structures. 
\end{example}

\begin{proposition}\label{prop:two central fiber}
Let $\pi \colon \mathcal X \rightarrow \Delta$ be a smooth projective morphism from a complex manifold $\mathcal X$
to the unit disc $\Delta\subset\mathbb C$. 
Suppose for any $t\in \Delta \backslash \{0\}$,
the fiber $\mathcal X_t=\pi^{-1}(t)$ is biholomorphic to the double Cayley Grassmannian $S$.
Then the central fiber $\mathcal X_0$ is biholomorphic to either $S$ or an equivariant compactification of the vector group $\mathbb G_a^{14}$.
\end{proposition}

\begin{proof}
The argument of the proof is the same as that for Theorem~1.1 of \cite{KP19}. 
First, the identity component $\Aut^0(S)$ of the automorphism group is isomorphic to $G_2 \times G_2$ by Lemma~16 of \cite{Ruzzi2010}. 
From Corollary~\ref{VMRT}, we can compute the Lie algebras of infinitesimal automorphisms of the affine cones of VMRTs: 
$\mathfrak{aut}(\widehat{\mathcal C_x(\mathcal X_0)}) = \mathfrak{aut}(\widehat{\mathcal C_s(S)}) = \mathfrak{aut}(\widehat{G_2^{ad}})
\cong \mathfrak g_2 \oplus \mathbb C$. 
Thus, for the cone structure $\mathcal C$ on a fiber $\mathcal X_t$ given by its VMRT 
we have equalities:  
$$\dim \mathfrak{aut}(S) + 1 
= \dim S+\dim \mathfrak{aut}(\widehat{\mathcal C_s})
= \dim \mathcal X_0 +\dim \mathfrak{aut}(\widehat{\mathcal C_x}).$$
Since the variety $\mathcal C_s(S)$ of minimal rational tangents of $S$ is irreducible smooth nondegenerate and not linear,
$\mathfrak{aut}(\widehat{\mathcal C_s})^{(k)} = 0$ for $k\geq 2$ 
by Proposition~1.1.2 of \cite{HwM05}.  
Furthermore, the classification of projective varieties with non-zero prolongation in \cite{FuHw18} implies that $\mathfrak{aut}(\widehat{\mathcal C_s})^{(k)}=0$ for all $k \geq 1$.

Because the Lie algebra $\mathfrak{aut}(\mathcal X_0)$ is isomorphic to the space $H^0(\mathcal X_0, T_{\mathcal X_0})$ of global sections of the tangent bundle $T_{\mathcal X_0}$,
we know $h^0(\mathcal X_0, T_{\mathcal X_0})=\dim \mathfrak{aut}(\mathcal X_0)$. 
Since the action of $\Aut(\mathcal X_0)$ preserves the VMRT-structure $\mathcal C$ on $\mathcal X_0$, 
we have an inclusion $\mathfrak{aut}(\mathcal X_0) \subset \mathfrak{aut}(\mathcal C, x)$. 
Hence, from Proposition~\ref{aut} we have inequalities
\begin{eqnarray*}
h^0(\mathcal X_0, T_{\mathcal X_0})
= \dim \mathfrak{aut}(\mathcal X_0)  & \leq & \dim \mathfrak{aut}(\mathcal C, x) \\
&\leq& \dim \mathcal X_0
+\dim \mathfrak{aut}(\widehat{\mathcal C_x}) \\
&=&\dim \mathfrak{aut}(S) +1 = h^0(S,T_S) +1.
\end{eqnarray*}
Now, recall that the Euler--Poincar\'e characteristic of the holomorphic tangent bundle $T_X$ on a Fano manifold $X$ is given by $\chi(X,T_X)=h^0(X,T_X)-h^1(X,T_X).$ 
Indeed, the Serre duality and the Kodaira--Nakano vanishing theorem imply that $H^i(X, T_X)=H^{n-i}(X, T^*_X \otimes K_X)^{\vee} = 0$ for $i\geq 2$.
Since the Euler--Poincar\'e characteristic is constant in a smooth family 
and $h^1(S, T_S)=0$ by Proposition~\ref{Local deformation rigidity of the double Cayley Grassmannian}, 
we conclude that $$h^1(\mathcal X_0, T_{\mathcal X_0}) = h^0(\mathcal X_0, T_{\mathcal X_0}) - h^0(S, T_S)\leq 1.$$

Suppose that the above equality holds. 
Then we have $\dim \mathfrak{aut}(\mathcal C, x) = \dim \mathcal X_0 + \dim \mathfrak{aut}(\widehat{\mathcal C_x})$, 
which implies that the isotrivial cone structure $\mathcal C$ given by VMRT on the central fiber $\mathcal X_0$ should be locally flat by Corollary~5.13 of \cite{FuHw}. 
Thus $\mathcal X_0$ is an equivariant compactification of the vector group $\mathbb G_a^{14}$ from Theorem~1.2 of \cite{FuHw2018-2}. 
\end{proof}

\begin{corollary}\label{cor: central fiber}
All assumptions are the same as above and assume that $\mathcal X_0$ is not biholomorphic to $S$. 
Then $\mathcal X_0$ is an equivariant compactification of the vector group $\mathbb G_a(\mathfrak g_2)$ in which $\mathfrak g_2$ is viewed as a vector group. 
We also have $\mathfrak{aut}(\mathcal X_0)=\mathbb C^{14} \oplus (\mathfrak g_2 \oplus \mathbb C)$ and $\Aut^0(\mathcal X_0)= \mathbb G_a(\mathfrak g_2) \rtimes (G_2 \times \mathbb C^*)$.
\end{corollary}

\begin{proof}
By Corollary~\ref{VMRT} and its proof, the VMRT of the central fiber $\mathcal X_0$ at a general point $x$ is projectively equivalent to the adjoint variety $G_2^{ad} \subset \mathbb P(\mathfrak g_2)$. 
If $\mathcal X_0$ is not biholomorphic to $S$, then by Proposition~\ref{prop:two central fiber} we see that $\mathcal X_0$ is an equivariant compactification of the vector group $\mathbb G_a(\mathfrak g_2)$ in which $\mathfrak g_2$ is viewed as a vector group.
Hence, $\mathfrak{aut}(\widehat{\mathcal C_x(\mathcal X_0)}) = \mathfrak g_2 \oplus \mathbb C$ such that $\mathfrak g_2$ acts on $T_x \mathcal X_0=\mathfrak g_2$ by the adjoint action and $\mathbb C$ acts on $T_x \mathcal X_0=\mathbb C^{14}$ by the dilation. 

Moreover, by Proposition~5.14 of \cite{FuHw} we have
$$\mathfrak{aut} (\mathcal X_0) =\mathfrak{aut}({\mathcal C,x}) =\mathbb C^{14} \oplus \mathfrak{aut}(\widehat{\mathcal C_x(\mathcal X_0)}),$$
which implies that $\mathfrak{aut}(\mathcal X_0)=\mathbb C^{14} \oplus (\mathfrak g_2 \oplus \mathbb C)$.
The evaluation map  $\mathfrak{aut} (\mathcal X_0) = H^0(\mathcal X_0, T_{\mathcal X_0}) \rightarrow T_x \mathcal X_0$ sends a global section $\sigma$ of tangent bundle on $\mathcal X_0$ to its value $\sigma(x)$. 
Hence, we have the following:
 $$0 \rightarrow \mathfrak{aut}(\widehat{\mathcal C_x(\mathcal X_0)}) \rightarrow \mathfrak{aut} (\mathcal X_0)  \rightarrow T_x \mathcal X_0 \rightarrow 0,$$ 
that is,
 $$0 \rightarrow \mathfrak g_2 \oplus \mathbb C \rightarrow \mathbb C^{14} \oplus (\mathfrak g_2 \oplus \mathbb C) \rightarrow \mathbb C^{14} \rightarrow 0.$$
Furthermore, since $\mathcal X_0$ is an equivariant compactification of the vector group $\mathbb G_a(\mathfrak g_2)$, we have $\Aut^0(\mathcal X_0)= \mathbb G_a(\mathfrak g_2) \rtimes (G_2 \times \mathbb C^*)$.
\end{proof}

\subsection{Algebraic moment polytopes}

\begin{figure}[h!]
\begin{minipage}[b]{.25 \textwidth}
\centering 
\begin{tikzpicture}
\clip (-2.5,-0.5) rectangle (4.7,6.5); 
\foreach \x  in {-8,-7.5,...,9}{
  \draw[help lines,thick,dotted]
    (\x,-8) -- (\x,9);
}
\begin{scope}[y=(60:1)]

\coordinate (pi1) at (0,1);
\coordinate (pi2) at (-1,2);
\coordinate (v1) at ($7*(pi1)$);
\coordinate (v2) at ($7/2*(pi2)$);
\coordinate (a1) at (1,0);
\coordinate (a2) at (-2,1);
\coordinate (a3) at ($3*(a1)+(a2)$);

\coordinate (Origin) at (0,0);
\coordinate (asum) at ($(a1)+(a2)$);
\coordinate (2rho) at (-2,6);

\coordinate (barycenter) at (1.759-49/15,98/15);

\foreach \x  in {-8,-7,...,9}{
  \draw[help lines,thick,dotted]
    (\x,-8) -- (\x,9)
    (-8,\x) -- (9,\x) 
    [rotate=60] (\x,-8) -- (\x,9) ;
}
\foreach \x  in {-8,-7,...,9}{
  \draw[help lines,thick,dotted]
    (\x,-8) -- (\x,9);
}

\fill (Origin) circle (2pt) node[below left] {0};
\fill (pi1) circle (2pt) node[right] {$\varpi_1$};
\fill (pi2) circle (2pt) node[right] {$\varpi_2$};
\fill (a1) circle (2pt) node[below] {$\alpha_1$};
\fill (a2) circle (2pt) node[above left] {$\alpha_2$};
\fill (a3) circle (2pt) node[below right] {$3\alpha_1+\alpha_2$};

\fill (asum) circle (2pt) node[above] {$\alpha_1+\alpha_2$};
\fill (2rho) circle (2pt) node[below] {$2\rho$};

\fill (v1) circle (2pt) node[right] {$7\varpi_1$};
\fill (v2) node[left] {$\frac{7}{2}\varpi_2$};

\draw[->,blue,thick](Origin)--(pi1);
\draw[->,blue,thick](Origin)--(pi2); 
\draw[->,blue,thick](Origin)--(a1);
\draw[->,blue,thick](Origin)--(a2);
\draw[->,blue,thick](Origin)--(a3); 
\draw[->,blue,thick](Origin)--(asum); 

\draw[thick](Origin)--(v1);
\draw[thick](Origin)--(v2);
\draw[thick](v1)--(v2);

\aeMarkRightAngle[size=6pt](Origin,v2,v1)

\end{scope}
\end{tikzpicture} 
\end{minipage}
\caption{Algebraic moment polytope of the anticanonical bundle of $S$.}
\label{algebraic moment polytope}

\vskip 1em 

\begin{minipage}[b]{.5 \textwidth}
\centering
\begin{tikzpicture}
\clip (-4.7,-4.1) rectangle (4.7,4); 
\foreach \x  in {-8,-7.5,...,9}{
  \draw[help lines,thick,dotted]
    (\x,-8) -- (\x,9);
}
\begin{scope}[y=(60:1)]

\coordinate (pi1) at (0,0.5);
\coordinate (pi2) at (-0.5,1);
\coordinate (v1) at ($7*(pi1)$);
\coordinate (v2) at ($7/2*(pi2)$);
\coordinate (a1) at (0.5,0);
\coordinate (a2) at (-1,0.5);
\coordinate (a3) at ($3*(a1)+(a2)$);
\coordinate (v3) at ($7*(a1)+7*(a2)$);
\coordinate (v4) at ($-7*(a1)$);
\coordinate (v5) at ($-7*(pi1)$);
\coordinate (v6) at ($-7*(a1)-7*(a2)$);
\coordinate (v7) at ($7*(a1)$);

\coordinate (Origin) at (0,0);
\coordinate (asum) at ($(a1)+(a2)$);
\coordinate (2rho) at (-1,3);

\foreach \x  in {-8,-7,...,9}{
  \draw[help lines,thick,dotted]
    (\x,-8) -- (\x,9)
    (-8,\x) -- (9,\x) 
    [rotate=60] (\x,-8) -- (\x,9) ;
}
\foreach \x  in {-8,-7,...,9}{
  \draw[help lines,thick,dotted]
    (\x,-8) -- (\x,9);
}

\fill (Origin) circle (2pt) node[below left] {0};
\fill (a1) circle (2pt) node[below] {$\alpha_1$};
\fill (a2) circle (2pt) node[above left] {$\alpha_2$};
\fill (2rho) circle (2pt) node[below] {$2\rho$};

\fill (v1) circle (2pt) node[right] {$7\varpi_1$};
\fill (v3) circle (2pt) node[above] {$7(\alpha_1+\alpha_2)$};
\fill (v4) circle (2pt) node[left] {$-7\alpha_1$};
\fill (v5) circle (2pt) node[left] {$-7\varpi_1$};
\fill (v6) circle (2pt) node[right] {$-7(\alpha_1+\alpha_2)$};
\fill (v7) circle (2pt) node[right] {$7\alpha_1$};

\draw[->,blue,thick](Origin)--(pi1);
\draw[->,blue,thick](Origin)--(pi2); 
\draw[->,blue,thick](Origin)--(a1);
\draw[->,blue,thick](Origin)--(a2);
\draw[->,blue,thick](Origin)--(a3); 
\draw[->,blue,thick](Origin)--(asum); 

\draw[thick](Origin)--(v1);
\draw[thick](Origin)--(v2);
\draw[thick](v1)--(v2);
\draw[thick](v2)--(v3);
\draw[thick](v3)--(v4);
\draw[thick](v4)--(v5);
\draw[thick](v5)--(v6);
\draw[thick](v6)--(v7);
\draw[thick](v7)--(v1);

\aeMarkRightAngle[size=6pt](Origin,v2,v1)
\end{scope}
\end{tikzpicture} 
\end{minipage}
\caption{Toric moment polytope of $\overline{T}$.}
\label{toric moment polytope}
\end{figure}

Let $o$ be the base point of the double Cayley Grassmannian $S$. 
We define $\overline{T}:=\overline{T.o} \subset S$ as the closure of the torus orbit, where $T$ is a 2-dimensional maximal torus of $G_2$. 
In order to obtain a description for the toric variety $\overline{T}$, let us recall the algebraic moment polytope for a polarized $G$-variety $(X, \mathcal L)$ introduced by Brion \cite{Brion87}, where the \emph{polarized} $G$-variety means the pair $(X, \mathcal L)$ of a projective $G$-variety $X$ and an ample line bundle $\mathcal L$ on $X$. 
Moreover, we say that $\mathcal L$ is \emph{$G$-linearized} if the $G$-action on $X$ can be lifted to $\mathcal L$ which is linear on fibers. 

\begin{definition} 
Let $\mathcal L$ be a $G$-linearized ample line bundle on a spherical $G$-variety $X$. 
The \emph{algebraic moment polytope} $\Delta(X, \mathcal L)$ of $(X, \mathcal L)$ is defined as the closure of $\displaystyle \bigcup_{k \in \mathbb N} \Delta_k / k$ in $\mathfrak X(T) \otimes \mathbb R$, 
where $\Delta_k$ is a finite set consisting of (dominant) weights $\lambda$ such that 
\begin{equation*}
H^0(X, \mathcal L^{\otimes k}) = \bigoplus_{\lambda \in \Delta_k} V_G(\lambda).
\end{equation*} 
Here, $V_G(\lambda)$ means the irreducible representation of $G$ with highest weight $\lambda$. 
\end{definition}

\begin{proposition}[Section~2 of \cite{AB04}] \label{toric amd moment polytopes} 
Let $(X, \mathcal L)$ be a polarized equivariant compactification of reductive group $G$, where $\mathcal L$ is $G$-linearized.  
If we denote by $Z$ the closure of a maximal torus $T$ of $G$ in $X$, then $Z$ is a toric manifold, equipped with an action of the Weyl group $W$, and the restriction $\mathcal L|_Z$ to $Z$ is a $W$-linearized ample toric line bundle on $Z$.
Furthermore, the toric polytope $P(Z, \mathcal L|_Z)$ is the union of
the images of the algebraic moment polytope $\Delta(X, \mathcal L)$ by $W$.
\end{proposition} 

\begin{lemma}\label{toric polytope of closure} 
The closure $\overline{T} \subset S$ of $T$ in the double Cayley Grassmannian $S$ is a smooth projective toric variety corresponding to a 2-dimensional polytope of which boundary is a regular hexagon (see Figure~\ref{toric moment polytope}). 
In particular, the toric variety $\overline{T}$ has Picard number 4. 
\end{lemma}

\begin{proof}
By Proposition~\ref{toric amd moment polytopes}, the toric moment polytope of $\overline{T}$ in Figure~\ref{toric moment polytope} is obtained by the Weyl group action from the algebraic moment polytope $\Delta(S, K_S^{-1})$ in Figure~\ref{algebraic moment polytope} (see Section~3.6 of \cite{LPY21}). 
From the description of the toric moment polytope in Figure~\ref{toric moment polytope}, we see that the compactification $\overline{T}$ in $S$ is a smooth projective toric surface of Picard number $4$. 
\end{proof}

\subsection{Proof of Theorem~\ref{Main theorem}}
We assume that $\mathcal X_0$ is not biholomorphic to the double Cayley Grassmannian $S$, that is, $\mathcal X_0$ is an equivariant compactification of the vector group $\mathbb G_a^{14}$ by Proposition~\ref{prop:two central fiber}. 
Because the Picard number of $\mathcal X_0$ is one, $\mathcal X_0$ is also Fano and $\mathcal X \rightarrow \triangle$ is a Fano family. 
So, we will use an argument similar to that in \cite{FuLi}. 

For the general fiber which is the double Cayley Grassmannian $S$, a compactification of the group $G_2$, the identity component of the automorphism group is isomorphic to $G_2 \times G_2$. 
Whereas for the central fiber $\Aut^0(\mathcal X_0)= \mathbb G_a(\mathfrak g_2) \rtimes (G_2 \times \mathbb C^*)$ due to Corollary~\ref{cor: central fiber}. 
From the action of $G_2$ on each fiber, there exists a $T$-action on the family $\pi \colon \mathcal X \rightarrow \Delta$, where $T$ is a 2-dimensional maximal torus of $G_2$. 
We consider a section $\sigma$ of $\pi \colon \mathcal X \rightarrow \Delta$ such that the image is contained in the fixed locus $\mathcal X^T$ of $T$. 
Choose a connected component $\mathcal Y$ of $\mathcal X^T$ containing $\sigma(\Delta)$, which is smooth by Corollary~5.4 of \cite{Fogarty}.

Let $o\in S$ be a general point, and consider the left $T$-action on $S$ at $o$. Then, since the set of $T$-fixed points of $G_2$ under the adjoint action is $T$, the general fiber $\mathcal Y_t$ is isomorphic to the closure $\overline{T}=\overline{T.o} \subset S$. On the other hand, since $\mathcal X_0$ is an equivariant compactification of the vector group $\mathbb G_a(\mathfrak g_2)$, $\mathcal Y_0$ is also an equivariant compactification of the vector group $\mathbb G_a(\mathfrak t)$, where $\mathfrak t$ is the Lie algebra of $T$. 
Hence, we have the smooth family $\varpi \colon \mathcal Y \rightarrow \Delta$ such that the fibers are smooth surfaces given by $\mathcal Y_t:=\overline{T.\sigma(t)} \subset \mathcal X_t$ for $t\neq 0$ and $\mathcal Y_0= \overline{\mathbb G_a(\mathfrak t)} \subset \mathcal X_0$. 
See Proposition 4.12 (2) in \cite{FuLi}. 

\begin{lemma}\label{Cor: Picard group}
For the smooth family $\varpi \colon \mathcal Y \rightarrow \Delta$ defined in the preceding paragraph, the Picard group of fibers is invariant, that is, $\Pic(\mathcal Y_t)=\Pic(\mathcal Y_0)$ for $t\neq 0$. 
In particular, the central fiber $\mathcal Y_0$ of $\varpi \colon \mathcal Y \rightarrow \Delta$ is a smooth surface of Picard number $4$.   
\end{lemma}

\begin{proof}
As this is a general statement, let's prove it for general situations. 
Consider a smooth family $\mathcal Y \rightarrow \triangle$ such that $\mathcal Y_t$ are Fano for all $t\neq 0$. 
Since $\mathcal Y_t$ is rationally connected, $\mathcal Y_0$ is also rationally connected. 
Then, we have $H^i(\mathcal Y_0, \mathcal O_{\mathcal Y_0}) = H^i(\mathcal Y_t, \mathcal O_{\mathcal Y_t})=0$ for all $i=1,2$. 
From the exponential sequence
$$0 \rightarrow \mathbb Z \rightarrow \mathcal O_{\mathcal Y_t} \rightarrow \mathcal O^*_{\mathcal Y_t} \rightarrow 0,$$
we have 
$$0=H^1(\mathcal Y_t, \mathcal O_{\mathcal Y_t}) \rightarrow H^1(\mathcal Y_t, \mathcal O^*_{\mathcal Y_t}) \rightarrow H^2(\mathcal Y_t, \mathbb Z) \rightarrow H^2(\mathcal Y_t, \mathcal O_{\mathcal Y_t})=0.$$
In particular, we have $H^2(\mathcal Y_t, \mathbb Z)=H^1(\mathcal Y_t, \mathcal O^*_{\mathcal Y_t})=\rm{Pic}(\mathcal Y_t)$ for all $t$. 
On the other hand, since $\mathcal Y_t$ are smooth, $H^2(\mathcal Y_t, \mathbb Z)$ are invariant for all $t$, hence we have $\Pic(\mathcal Y_t)=\Pic(\mathcal Y_0)$ for all $t$. 

For the smooth family $\varpi \colon \mathcal Y \rightarrow \Delta$, by Lemma~\ref{toric polytope of closure} the general fiber $\mathcal Y_t=\overline{T}$ is a smooth Fano toric surface of Picard number 4. 
Hence, the central fiber $\mathcal Y_0$ should be also a smooth surface of Picard number $4$.
\end{proof}

Let $W(G_2)=N_{G_2}(T)/T$ be the Weyl group of $G_2$. 
Since we have the action of $G_2$ on the whole family $\mathcal X / \triangle$ from Corollary~\ref{cor: central fiber} and $\mathcal Y$ is a connected component of $\mathcal X^T$, it induces an action of the Weyl group $W(G_2)$ on the smooth family $\mathcal Y / \triangle$. 
In particular, $W(G_2)$ acts on the boundary $\partial \mathcal Y_t$ for each $t$.

\begin{lemma}\label{Orbit} 
The Weyl group $W(G_2)$ of $G_2$ acts transitively on the set of irreducible components of the boundary divisor $\partial \mathcal Y_t = \mathcal Y_t \backslash T$ for all $t\neq 0$, that is, $W(G_2)$ has one orbit.
\end{lemma}

\begin{proof}
Let $W:=W(G_2)$. 
Let $\epsilon_1$, $\epsilon_2$, $\epsilon_3$ be vectors of the dual $\mathfrak t^*(G_2)$ of the Cartan subalgebra $\mathfrak t$ such that $\sum \epsilon_i = 0$, $(\epsilon_i,\epsilon_i)=3/4$ and $(\epsilon_i,\epsilon_j)=-1/4$ for $i\neq j$. 
With these vectors, the simple roots are $\alpha_1=-\epsilon_2$ and $\alpha_2=\epsilon_2-\epsilon_3$, the fundamental weights are $\varpi_1=\epsilon_1$ and $\varpi_2=\epsilon_1-\epsilon_3$ and the Weyl group $W$ is given by $\{\pm I\} \times \mathfrak S_3$, where $I$ is the identity map and the symmetric group $\mathfrak S_3$ of order 6 acts on the set $\{\epsilon_1, \epsilon_2, \epsilon_3 \}$ as permutations. 
The $W$-orbit of the fundamental weight $\varpi_1$ is the set $\{ \pm \epsilon_i : =1,2,3 \}$ and the $W$-orbit of the fundamental weight $\varpi_2$ is the set $\{ \epsilon_i-\epsilon_j : 1 \leq i \neq j \leq 3 \}$. 
Since the divisor corresponding to $\varpi_1$ disappears by blowing down from the wonderful compactification of $G_2$, an irreducible component of the boundary divisor $\partial \overline{T}$ corresponds to a unique ray generated by an element in $W.\varpi_2=\{ \epsilon_i-\epsilon_j  : 1 \leq i \neq j \leq 3 \}$. 
In conclusion, $W$ acts transitively on the set of irreducible components of the boundary divisor $\partial \mathcal Y_t$ for all $t\neq 0$.
\end{proof}

By Lemma~\ref{Cor: Picard group}, we have $\Pic(\mathcal Y_t)=\Pic(\mathcal Y_0)$ for all $t\neq 0$, so that $\Pic(\mathcal Y/\triangle)=\Pic(\mathcal Y_t)$ for all $t \in \triangle$. 
Since the actions of $W(G_2)$ on the maximal torus $T \subset G_2$ and on its Lie algebra $\mathfrak t$ are the natural ones, the induced action of $W(G_2)$ on $\mathcal Y$ gives identifications $\Pic(\mathcal Y_t)^{W(G_2)}=\Pic(\mathcal Y/\triangle)^{W(G_2)}=\Pic(\mathcal Y_0)^{W(G_2)}$. 
By Lemma~\ref{Orbit}, 
the rank of $W(G_2)$-fixed Picard group of $\mathcal Y_t$ is one for $t \neq 0$. 
Hence, the $W(G_2)$-fixed Picard group of $\mathcal Y_0$ has also rank one. 

Let $D_1, \cdots, D_r$ be the irreducible components of the boundary divisor $\partial \mathcal Y_0 = \mathcal Y_0 \backslash \mathbb G_a(\mathfrak t)$. 
Then $\Pic(\mathcal Y_0)^{W(G_2)} \otimes \mathbb Q$ is generated by $\{\sum_{j \in J_i} D_{i, j}\}_{i \in I}$, where the elements in $\{D_{i, j} : j \in J_i \}$ are in the same $W(G_2)$-orbit by Lemma~4.16  of \cite{FuLi}, hence $D_1+ \cdots +D_r$ is the generator of $W(G_2)$-fixed Picard group of $\mathcal Y_0$. 
Otherwise, the $W(G_2)$-fixed Picard group of $\mathcal Y_0$ would have rank greater than one because $D_1, \cdots, D_r$ generate freely the Picard group of $\mathcal Y_0$ by Theorem~2.5 in \cite{HT99}. 
In conclusion, $W(G_2)$ acts transitively on the set of irreducible components of $\partial \mathcal Y_0$. 
However, a vector group compactification $\mathcal Y_0$ of $\mathbb G_a^2$ with this property is isomorphic to either $\mathbb P^2$ or $\mathbb P^1 \times \mathbb P^1$ by Remark~4.4 of \cite{FuLi} (see also Example~\ref{ex: rank 2 vector group}). 
This contradicts Lemma~\ref{Cor: Picard group} saying the Picard number of $\mathcal Y_0$ is 4. 
Therefore, the central fiber $\mathcal X_0$ is also biholomorphic to the general fiber $S$.
\qed

\vskip 3em 

\providecommand{\bysame}{\leavevmode\hbox to3em{\hrulefill}\thinspace}
\providecommand{\MR}{\relax\ifhmode\unskip\space\fi MR }
\providecommand{\MRhref}[2]{%
  \href{http://www.ams.org/mathscinet-getitem?mr=#1}{#2}
}
\providecommand{\href}[2]{#2}

\end{document}